\documentclass[10pt, article]{amsart}
\usepackage{geometry} 
\usepackage{ae} 
\usepackage[T1]{fontenc}
\usepackage[cp1250]{inputenc}
\usepackage{amsmath}
\usepackage{amssymb, amsfonts,amscd,verbatim}
\usepackage{MnSymbol}
\usepackage{mathrsfs} 
\usepackage{xypic}

\usepackage[normalem]{ulem}
\usepackage{hyperref}
\usepackage{indentfirst}
\usepackage{latexsym}
\input xy
\xyoption{all}

\usepackage{amsmath}    

\theoremstyle{plain}
\newtheorem{Pocz}{Poczatek}[section]
\newtheorem{Proposition}[Pocz]{Proposition}

\newtheorem{Theorem}[Pocz]{Theorem}
\newtheorem{Corollary}[Pocz]{Corollary}

\newtheorem{Lemma}[Pocz]{Lemma}

\newtheorem{Question}[Pocz]{Question}

\newtheorem{Example}[Pocz]{Example}

\theoremstyle{definition}
\newtheorem{Definition}[Pocz]{Definition}

\theoremstyle{remark}
\newtheorem{Remark}[Pocz]{Remark}

\def\asdim{\mathrm{asdim}}

\DeclareMathOperator*{\st}{st}
\numberwithin{equation}{section}

\author{Jerzy~Dydak}
\address{University of Tennessee, Knoxville, USA}
\email{jdydak@utk.edu}
\author{Thomas ~ Weighill}
\address{University of Tennessee, Knoxville, USA}
\email{tweighil@vols.utk.edu}

\title[Monotone-light factorizations in coarse geometry
]%
  {Monotone-light factorizations in coarse geometry}

\date{ \today
}


\begin{document}


\begin{abstract}
We introduce large scale analogues of topological monotone and light maps, which we call coarsely monotone and coarsely light maps respectively. We show that these two classes of maps constitute a factorization system on the coarse category. We also show how coarsely monotone maps arise from a reflection in a similar way to classically monotone maps, and prove that coarsely monotone maps are stable under those pullbacks which exist in the coarse category. For the case of maps between proper metric spaces, we exhibit some connections between the coarse and classical notions of monotone and light using the Higson corona. Finally, we look at some coarse properties which are preserved by coarsely light maps such as finite asymptotic dimension and exactness, and make some remarks on the situation for groups and group homomorphisms.
\end{abstract}

\maketitle

\section{Introduction}
Recall that continuous map from a compact Hausdorff space $X$ to a compact Hausdorff space $Y$ is called \textbf{monotone} if it is surjective and each of its fibres is connected, and is called \textbf{light} if each of its fibres is totally disconnected (see for example~\cite{Whyburn}). Eilenberg showed in \cite{Eilen} that every continuous map $f$ between compact metric spaces factorizes as $f = me$, where $m$ is light and $e$ is monotone (in fact, the result holds more generally for compact Hausdorff spaces, see~\cite{CarJanKelPar}). Moreover, this factorization satisfies a universal property, namely that for any commutative diagram
$$ \xymatrix{
\bullet \ar[r]^e \ar[d]_u & \bullet  \ar[r]^m & \bullet \ar[d]^v \\
\bullet \ar[r]_{e'} & \bullet \ar[r]_{m'} & \bullet \\
}
$$
 where the arrows are continuous maps and the objects are compact Hausdorff spaces, with $e'$ monotone and $m'$ light, 
there is a unique continuous map $h$ making the diagram commute:
$$ \xymatrix{
\bullet \ar[r]^e \ar[d]_u & \bullet \ar@{..>}[d]^h \ar[r]^m & \bullet \ar[d]^v \\
\bullet \ar[r]_{e'} & \bullet \ar[r]_{m'} & \bullet \\
}
$$
In the language of category theory, this is to say that the classes of monotone maps and light maps constitute a \textbf{factorization system}~\cite{FreKel} on the category of compact Hausdorff spaces and continuous maps.

In coarse geometry one is interested in the large scale properties of metric (or more general) spaces, or in other words, properties of spaces ``as viewed from a great distance'' \cite{NowakYu}. For example, the metric spaces $\mathbb{R}$ and $\mathbb{Z}$ are ``coarsely equivalent'', that is, isomorphic in the coarse category, despite being very different as topological spaces. The motivation to study large scale behaviour comes mostly from geometric group theory and index theory (see for example \cite{Gromov93} and \cite{RoeIndex} respectively). Many classical topological notions have large scale analogues; for example, Gromov introduced the notion of asymptotic dimension in \cite{Gromov93}, which, when defined in terms of covers of a space, is clearly analogous to the covering dimension of a topological space. It was later shown that for proper metric spaces, the asymptotic dimension coincides with the (topological) covering dimension of the Higson corona (a topological space that captures large scale behaviour of a metric space) whenever the former is finite~\cite{Dranishnikov00}. This gives another connection between the large scale and topological notions of dimension.

In this paper, we introduce large scale analogues of the topological monotone and light maps mentioned above, to which we give the names coarsely monotone and coarsely light maps respectively. These classes of maps will turn out to constitute a factorization system on the coarse category (defined in the next section). A large part of the paper is devoted to making some connections between the topological and large scale notions of monotone and light. We do so in two ways. Firstly, we examine these classes of maps from a categorical perspective inspired by the results in \cite{CarJanKelPar}. Secondly, we make some connections using the Higson corona in the case when the large scale spaces involved are proper metric spaces. Coarsely light maps generalize both coarse embeddings and coarsely n-to-1 maps; we prove that coarsely light maps preserve certain coarse properties such as finite asymptotic dimension and Yu's Property A in a similar way to these classes of maps. In the final section of the paper, we make some remarks on coarsely monotone and light maps between groups.

The main goal of this paper is to introduce two interesting classes of maps between large scale spaces and study some of their properties. Along the way, however, we also investigate some of the structure of the coarse category and apply some basic categorical arguments to large scale spaces and maps between them. It would be interesting to see what other categorical notions turn out to be useful in the study of large scale spaces.

\section{Preliminaries}
We recall some basic terminology from \cite{DH}. Let $X$ be a set. Recall that the \textbf{star} $\st(B, \mathcal{U})$ of a subset $B$ of $X$ with respect to a family $\mathcal{U}$ of subsets of $X$ is the union of those elements of $\mathcal{U}$ that intersect $B$. More generally, for two families $\mathcal{B}$ and $\mathcal{U}$ of subsets of X, $\st(\mathcal{B}, \mathcal{U})$ is the family $\{\st(B, \mathcal{U}) \mid B \in \mathcal{B}\}$.

\begin{Definition}
A \textbf{large scale structure} $\mathcal{L}$ on a set $X$ is a nonempty
set of families $\mathcal{B}$ of subsets of $X$ (which we call the \textbf{uniformly bounded families} in $X$) satisfying the following conditions:
\begin{itemize}
\item[(1)] $\mathcal{B}_1 \in \mathcal{L}$ implies $\mathcal{B}_2 \in \mathcal{L}$ if each element of $\mathcal{B}_2$ consisting of more than one point is contained in some element of $\mathcal{B}_1$.
\item[(2)] $\mathcal{B}_1, \mathcal{B}_2 \in \mathcal{L}$ implies $\st(\mathcal{B}_1, \mathcal{B}_2) \in \mathcal{L}$.
\end{itemize}
\end{Definition}

\begin{Remark}
Note that any uniformly bounded family $\mathcal{U}$ can be extended to a cover which is also uniformly bounded by adding singleton sets to the family (we call this cover the \textbf{trivial extension} of $\mathcal{U}$), so we will often assume that a given family is in fact a cover for convenience. Note that a family $\mathcal{U}$ of subsets of $X$ refines $\st(\mathcal{U}, \mathcal{V})$ for any cover $\mathcal{V}$ of $X$.
\end{Remark}

By a large scale space (or ls-space for short), we mean a set equipped with a large scale structure. A subset of a large scale space $X$ is called \textbf{bounded} if it is an element of some uniformly bounded family in $X$. A classical example of a large scale space is as follows. Let $(X, d)$ be an $\infty$-metric space. Define the uniformly bounded families in $X$ to be all those families $\mathcal{U}$ for which there is a $M > 0$ such that every element of $\mathcal{U}$ has diameter at most $M$. In fact, a large scale structure on $X$ arises in this way from an $\infty$-metric on $X$ if and only if it is countably generated~\cite{DH}. We call such a large scale structure \textbf{metrizable}. 

Given a set map $f: X \rightarrow Y$ from an ls-space $X$ to an ls-space $Y$, we say that $f$ is \textbf{large scale continuous} or \textbf{ls-continuous} if for every uniformly bounded family $\mathcal{U}$ in $X$, the family
$$
f(\mathcal{U}) = \{f(U) \mid U \in \mathcal{U} \}
$$
is uniformly bounded in $Y$. Given two set maps $f,g: X \rightarrow Y$, we say that $f$ and $g$ are \textbf{close} and write $f \sim g$ if the family
$$
\{\{f(x), g(x)\} \mid x \in X \}
$$
is uniformly bounded, in which case we say that this family (or any uniformly bounded family which it refines) \textbf{witnesses} the closeness of $f$ and $g$. 

Certain types of ls-continuous maps are worth mentioning.  Recall that a map $f: X \rightarrow Y$ between ls-spaces is called \textbf{coarsely surjective} if there is a uniformly bounded family $\mathcal{U}$ in $Y$ such that $Y \subseteq \st(f(X), \mathcal{U})$. An ls-continuous map $f$ is called a \textbf{coarse equivalence} if there is an ls-continuous map $f'$ in the other direction such that $ff'$ and $f'f$ are both close to the identity. An ls-continuous map $f:X \rightarrow Y$ is called a \textbf{coarse embedding} if for every uniformly bounded family $\mathcal{U}$ in $Y$, $f^{-1}(\mathcal{U}) = \{f^{-1}(U) \mid U \in \mathcal{U} \}$ is uniformly bounded in $X$. It is easy to check that an ls-continuous map is a coarse equivalence if and only if it is coarsely surjective and a coarse embedding. 

We are now ready to introduce the category on which we will construct a factorization system. In the present paper, by the \textbf{coarse category} we mean the category whose objects are large scale spaces and whose morphisms are equivalence classes of ls-continuous maps under the closeness relation. Note that this differs from Roe's coarse category in~\cite{RoeIndex}, where the maps are further required to be proper (an ls-continuous map is \textbf{proper} if the inverse image of every bounded set is bounded), although similar results will hold for this category as well (see Remark \ref{RoeCoarse}). It is easy to check that composition is well-defined in the coarse category, that is, if $f \sim g$ and $h \sim j$, then $hf \sim jg$ whenever these composites are defined. Note that the isomorphisms in the coarse category are represented by coarse equivalences.

For a set $X$, a family $\mathcal{U}$ of subsets of $X$ and $x, x' \in X$, we write $x \mathcal{U} x'$ to mean that there is an element of $\mathcal{U}$ containing both $x$ and $x'$. We say that $x$ and $x'$ are \textbf{$\mathcal{U}$-connected} if there is a finite sequence $x = x_1,\ x_2,\ \ldots,\ x_k = x'$ of elements of $X$ with $x_i \mathcal{U} x_{i+1}$. Equivalence classes under the relation ``$x$ is $\mathcal{U}$-connected to $x'$'' will be called $\mathcal{U}$-\textbf{components}. For $\mathcal{U}$ and $\mathcal{V}$ two families of subsets of $X$, we write $\mathcal{U} \leq \mathcal{V}$ in case $\mathcal{U}$ refines $\mathcal{V}$, in which case we also say that $\mathcal{V}$ \textbf{coarsens} $\mathcal{U}$.

\section{Coarsely light maps}\label{light}
In this section we introduce the large scale analogue of topological light maps.

\begin{Definition}
Let $f: X \rightarrow Y$ be an ls-continuous map between ls-spaces. For every pair of uniformly bounded families $\mathcal{U}$ in $X$ and $\mathcal{V}$ in $Y$, denote by $c(\mathcal{U}, f, \mathcal{V})$ the family of subsets consisting of all $\mathcal{U}$-components of elements of $f^{-1}(\mathcal{V})$. The closure under refinement of the set of all such $c(\mathcal{U}, f, \mathcal{V})$ is called the \textbf{light structure on $X$ with respect to $f$}. 
\end{Definition}

\begin{Proposition}
Let $f: X \rightarrow Y$ be an ls-continuous map between ls-spaces. Then the light structure on $X$ with respect to $f$ is an ls-structure which contains the ls-structure on $X$.
\end{Proposition}
\begin{proof}
Let $c(\mathcal{U}, f, \mathcal{V})$ and $c(\mathcal{U}', f, \mathcal{V}')$ be two elements of the light structure. We may suppose that $\mathcal{U}$, $\mathcal{U}'$, $\mathcal{V}$ and $\mathcal{V}'$ are covers and that $\mathcal{V}$ and $\mathcal{V}'$ coarsen $f(\mathcal{U})$ and $f(\mathcal{U}')$ respectively. It is easy to check that 
$$\st(c(\mathcal{U}, f, \mathcal{V}), c(\mathcal{U}', f, \mathcal{V}')) \leq c(\st(\mathcal{U}, \mathcal{U}'), f, \st(\mathcal{V}, \mathcal{V}')).$$
It follows that the light structure is an ls-structure. To see that it contains the ls-structure on $X$, note that for a uniformly bounded cover $\mathcal{U}$ of $X$, we have $\mathcal{U} \leq c(\mathcal{U}, f, f(\mathcal{U}))$.
\end{proof}

\begin{Proposition}
If $f: X \rightarrow Y$ is a map between metrizable coarse spaces, then the light structure on $X$ with respect to $f$ is also metrizable. 
\end{Proposition}
\begin{proof}
Since the ls-structures on $X$ and $Y$ are generated by metrics, we may assume that there are countable families $(\mathcal{U}_i)_{i \in \mathbb{N}}$ and $(\mathcal{V}_i)_{i \in \mathbb{N}}$ of uniformly bounded covers of $X$ and $Y$ respectively such that any uniformly bounded cover of $X$ refines some $\mathcal{U}_i$ and any uniformly bounded cover of $Y$ refines some $\mathcal{V}_i$. It follows that the light structure on $X$ with respect to $f$ is generated by the countable family $(c(\mathcal{U}_i, f, \mathcal{V}_j))_{i,j \in \mathbb{N}}$.
\end{proof}

\begin{Definition}
We say that an ls-continuous map $f: X \rightarrow Y$ is \textbf{coarsely light} if the light structure on $X$ with respect to $f$ coincides with the ls-structure on $X$.
\end{Definition}

If $X$ is an ls-space, and $A \subseteq X$ is a subset, then there is a natural ls-structure on $A$ induced by $X$, namely, all those families in $A$ which are uniformly bounded as families in $X$. We will call such an $A$ together with the induced structure a \textbf{subspace} of $X$. Given a collection $(A_\alpha)_{\alpha \in I}$ of subspaces of an ls-space $X$, let $\bigsqcup A_\alpha$ be the disjoint union of the $A_\alpha$, with the ls-structure given by all families $\mathcal{U}$ which satisfy the following conditions:
\begin{itemize}
\item[(1)] the image of $\mathcal{U}$ under the obvious map $p: \bigsqcup A_\alpha \rightarrow X$ is a uniformly bounded family;
\item [(2)] each member of $\mathcal{U}$ intersects at most one of the $A_\alpha$.
\end{itemize}
Using the above construction, we can formulate a generalization of the notion of uniform asymptotic dimension of subspaces of a metric space given in~\cite{BellDran}. A family $(A_\alpha)_{\alpha \in I}$ of subspaces of an ls-space $X$ \textbf{satisfies the inequality $\asdim \leq n$ uniformly} if $\asdim \bigsqcup A_\alpha \leq n$, that is, for every uniformly bounded cover $\mathcal{U}$ of $\bigsqcup A_\alpha$, there is a uniformly bounded cover $\mathcal{V}$ of $\bigsqcup A_\alpha$ which coarsens $\mathcal{U}$ and which has point multiplicity at most $n+1$. In particular, a space $X$ is of asymptotic dimension less than $n$ iff $\{X\}$ satisfies $\asdim \leq n$ uniformly. It is easy to see that an ls-space is of asymptotic dimension $0$ if and only if for every uniformly bounded family $\mathcal{U}$, the $\mathcal{U}$-components of $X$ form a uniformly bounded family.

\begin{Proposition}\label{lightequiv}
Let $f: X\rightarrow Y$ be an ls-continuous map. Then the following are equivalent:
\begin{itemize}
\item[(a)] $f$ is coarsely light,
\item[(b)] for any uniformly bounded cover $\mathcal{V}$ of $Y$, the family of subspaces $f^{-1}(\mathcal{V}) = \{f^{-1}(V) \mid V \in \mathcal{V}\}$ satisfies the inequality $\asdim \leq 0$ uniformly.
\end{itemize}
If $Y$ is a metric space with the induced ls-structure, then the above are further equivalent to
\begin{itemize}
\item[(c)] $\asdim f = 0$ in the sense of~\cite{BrodDyd}, that is, for every subspace $B \subseteq Y$ with $\asdim (B) = 0$, $\asdim f^{-1}(B) = 0$. 
\end{itemize}
\end{Proposition}
\begin{proof}
(a) $\Leftrightarrow$ (b): To say that $f$ is coarsely light is precisely to say that the $\mathcal{U}$-components of the elements of $f^{-1}(\mathcal{V})$ form a uniformly bounded family for any uniformly bounded families $\mathcal{U}$ in $X$ and $\mathcal{V}$ of $Y$. This is clearly equivalent to (b). 

(a) $\Rightarrow$ (c): Suppose $B \subseteq Y$ has asymptotic dimension $0$, and let $\mathcal{U}$ be a uniformly bounded cover of $f^{-1}(B)$. Consider the uniformly bounded cover $f(\mathcal{U})$ of $B$. By hypothesis, the $f(\mathcal{U})$-components of $B$ form a uniformly bounded family $\mathcal{V}$. Thus the $\mathcal{U}$-components of the family $f^{-1}(\mathcal{V})$ form a uniformly bounded cover of $f^{-1}(B)$. But this cover is precisely the set of $\mathcal{U}$-components of $f^{-1}(B)$.

(c) $\Rightarrow$ (a): If $f$ is not coarsely light, then there is a family $c(\mathcal{U}, f, \mathcal{V})$ which is not uniformly bounded, with $\mathcal{U}$ and $\mathcal{V}$ uniformly bounded families in $X$ and $Y$ respectively. In particular, there is a sequence $V_1, V_2, \ldots $ of elements of $\mathcal{V}$ such that $f^{-1}(V_i)$ has a $\mathcal{U}$-component of diameter greater than $i$. If there is a bounded set $K$ in $Y$ containing infinitely many of the $V_i$, then $\asdim K = 0$, but $f^{-1}(K)$ has an unbounded $\mathcal{U}$-component, which contradicts (c). Thus every bounded set $K$ in $Y$ contains only finitely many of the $V_i$. Since the $V_i$ are uniformly bounded, we may choose a subsequence $W_i$ such that $d(W_i, W_{i+1}) > i$. The union $W = \bigcup W_i$ is clearly of asympotic dimension $0$, but its inverse image has an unbounded $\mathcal{U}$-component, which once again contradicts (c).
\end{proof}

\begin{Example}
The following are examples of coarsely light maps:
\begin{itemize}
\item any ls-continuous map $f: X \rightarrow Y$ where $\asdim X = 0$.
\item any ls-continuous map $f: X \rightarrow Y$ where $Y$ has bounded geometry (i.e.~where the elements of any uniformly bounded cover have bounded cardinality) and $$\mathsf{sup}\{|f^{-1}(y)| \mid y \in Y\} < \infty.$$
\item any coarse embedding, and in particular any coarse equivalence.
\end{itemize}
\end{Example}

Recall that an ls-continuous map $f: X \rightarrow Y$ is called \textbf{coarsely n-to-1} \cite{AustinThesis, MiyVirk} if for every uniformly bounded cover $\mathcal{V}$ of $Y$ there is a uniformly bounded cover $\mathcal{U}$ of $X$ such that each element of $f^{-1}(\mathcal{V})$ is contained in the union of $n$ elements of $\mathcal{U}$. 

\begin{Proposition}
If $f: X \rightarrow Y$ is coarsely n-to-1, then it is coarsely light.
\end{Proposition}
\begin{proof}
Let $\mathcal{U}$ and $\mathcal{V}$ be uniformly bounded covers of $X$ and $Y$ respectively. We may assume that $\mathcal{U}$ is large enough so that every element of $f^{-1}(\mathcal{V})$ is contained in a union of $n$ elements of $\mathcal{U}$. It follows that $c(\mathcal{U}, f, \mathcal{V})$ refines $\mathcal{U}^{n-1}$, the star of $\mathcal{U}$ with itself $n-1$ times, so it is uniformly bounded.
\end{proof}

\begin{Lemma}
If $f: X \rightarrow Y$ is coarsely light and is close to $g$, then $g$ is coarsely light. 
\end{Lemma}
\begin{proof}
This follows from the fact that for any uniformly bounded cover $\mathcal{V}$ of $Y$, there is a uniformly bounded cover $\mathcal{U}$ of $Y$ such that $g^{-1}(\mathcal{V}) \leq f^{-1}(\mathcal{U})$.
\end{proof}

By the lemma above, it makes sense to speak of the class of coarsely light maps in the coarse category: a morphism in the coarse category is coarsely light if and only if one (and hence all) of its representatives are coarsely light. The following lemma shows that coarsely light maps form a subcategory of the coarse category, i.e.~that they are closed under composition.

\begin{Lemma}\label{comp}
If $f: X \rightarrow Y$ and $g: Y \rightarrow Z$ are coarsely light maps, then $gf$ is a coarsely light map.
\end{Lemma}
\begin{proof}
Let $\mathcal{V}$ be a uniformly bounded cover of $Z$ and $\mathcal{U}$ a uniformly bounded cover of $X$. Since $f$ and $g$ are coarsely light, $c(f(\mathcal{U}), g, \mathcal{V})$ is uniformly bounded, as is $c(\mathcal{U}, f, c(f(\mathcal{U}), g, \mathcal{V}))$. Since $c(\mathcal{U}, gf, \mathcal{V})$ is refinement of this cover, it is uniformly bounded.
\end{proof}

Given any ls-continuous map $f: X \rightarrow Y$ of coarse spaces, let $X_f$ denote $X$ with the light structure with respect to $f$. Then $f$ factorises as 
\begin{equation} \label{fact1}
\begin{gathered}
\xymatrix{
X \ar[rd]_e \ar[rr]^f & & Y \\
& X_f \ar[ru]_{f'} \\
}
\end{gathered}
\end{equation}
where $e$ is the identity set map. One checks that $f'$ is ls-continuous and coarsely light; we will call $f'$ the \textbf{light-part of $f$}. This factorization satisfies a universal property, as the following lemma shows.

\begin{Lemma}\label{rightM}
Let $f = f'e$ be the factorization as above. Given any diagram of solid arrows below consisting of ls-continuous maps which commutes up to closeness and in which $n$ is a coarsely light map, there is a unique-up-to-closeness ls-continuous map $g$ making the diagram commute up to closeness.
\begin{equation} \label{rightfact}
\begin{gathered}
\xymatrix@=40pt{
X \ar@/_2pc/[dd]_f \ar[d]^e \ar[rd]^{e'} \\
X_f \ar[d]^{f'} \ar@{..>}[r]_g& W \ar[d]^{n} \\
Y \ar[r]_h & Z\\
}
\end{gathered}
\end{equation}
\end{Lemma}
\begin{proof}
Since $e$ is the identity as a set map, the map $g$, if it exists, is clearly unique-up-to-closeness. Define a $g: X_f \rightarrow W$ to be the same as $e'$ at the level of underlying sets. It that remains to show that $g$ so defined is ls-continuous. Consider a uniformly bounded family in $X_f$, which we may suppose to be of the form $c(\mathcal{U}, f,  \mathcal{V})$. If $x$ and $x'$ are in an element $U$ of $\mathcal{U}$, then $g(x)$ and $g(x')$ are in the subset $g(U) = e'(U) \in e'(\mathcal{U})$. Moreover, if $f'(x) = f(x)$ and $f'(x') = f'(x')$ are both in an element $V \in \mathcal{V}$, then $hf'(x)$ and $hf'(x')$ are both in $h(V)\in h(\mathcal{V})$. Let $\mathcal{T}$ be the uniformly bounded cover which witnesses the closeness of $hf$ and $ne'$. Then under the assumptions on $x$ and $x'$, $ng(x)$ and $ng(x')$ are in some $V' \in \mathcal{V}' = \st(h(\mathcal{V}), \mathcal{T})$. In other words, $g(c(\mathcal{U}, f, \mathcal{V})) \leq c(e'(\mathcal{U}), n, \mathcal{V}')$. Since $n$ is coarsely light, this second cover is uniformly bounded, so $g$ is ls-continuous as required.
\end{proof}

\section{Coarsely monotone maps and monotone-light factorizations}
We are now ready to define coarsely monotone maps. Let $f: X \rightarrow Y$ be an ls-continuous map and consider its factorization $f = f'e$ as in (\ref{fact1}). We say that $f$ is \textbf{coarsely monotone} if $f'$ (i.e.~the light-part of $f$) is a coarse equivalence. Since coarse equivalences are always light, it is easy to see that any coarse equivalence is also monotone. The following lemma is easy to show.

\begin{Lemma}\label{mono}
An ls-continuous map $f: X \rightarrow Y$ is coarsely monotone if and only if it is coarsely surjective and for every uniformly bounded cover $\mathcal{V}$ of $Y$, there is a uniformly bounded family $\mathcal{U}$ in $X$ and a family $\mathcal{T}$ of $\mathcal{U}$-connected subsets of $X$ which coarsens $f^{-1}(\mathcal{V})$ such that $f(\mathcal{T})$ is a uniformly bounded family in $Y$.
\end{Lemma}

The following lemma shows that it makes sense to speak of coarsely monotone maps in the coarse category.

\begin{Lemma}
If $f \sim g$ and $f$ is coarsely monotone, then $g$ is coarsely monotone.
\end{Lemma}
\begin{proof}
The light structures induced by $g$ and $f$ are the same, so the light-parts of $g$ and $f$ are close. The result follows.
\end{proof}

Let $\mathbb{C}$ be a general category and let $(\mathcal{E}, \mathcal{M})$ be a pair of classes of morphisms in $\mathbb{C}$. Recall that the pair $(\mathcal{E}, \mathcal{M})$ is said to constitute a factorization system~\cite{FreKel} if the following conditions are satisfied:
\begin{itemize}
\item[(1)] each of $\mathcal{E}$ and $\mathcal{M}$ contains the isomorphisms and is closed under composition;
\item[(2)] every morphism $f$ in $\mathbb{C}$ can be written as $f = me$ with $m \in \mathcal{M}$ and $e \in \mathcal{E}$;
\item[(3)] given any commutative diagram of solid arrows below, with $e, e' \in \mathcal{E}$, $m, m' \in \mathcal{M}$, there is a unique morphism $h$ making the diagram commute:
$$ \xymatrix{
\bullet \ar[r]^e \ar[d]_u & \bullet \ar@{..>}[d]^h \ar[r]^m & \bullet \ar[d]^v \\
\bullet \ar[r]_{e'} & \bullet \ar[r]_{m'} & \bullet \\
}
$$
\end{itemize}

Note that in (3), if $u$ and $v$ are isomorphisms, then so is $h$. Let $\mathsf{CMon}$ be the class of (equivalence classes of) coarsely monotone maps and $\mathsf{CLight}$ be the class of (equivalence classes of) coarsely light maps in the coarse category. 

\begin{Theorem} \label{factsys}
The pair $(\mathsf{CMon}, \mathsf{CLight})$ constitutes a factorization system on the coarse category. In particular, every ls-continuous map $f$ factorizes as $f = f'e$ where $f'$ is coarsely light and $e$ is coarsely monotone.
\end{Theorem}
\begin{proof}
Let $f$ be an ls-continuous map, and let $f = f'e$ be the factorization of $f$ as in (\ref{fact1}). Clearly $e$ is coarsely monotone, so this proves (2) in the definition of a factorization system. Condition (3) is an easy consequence of Lemma~\ref{rightM}. Thus, after Lemma~\ref{comp}, all that remains to be shown is that coarsely monotone maps are closed under composition. This in fact follows from (2), (3) and Lemma~\ref{comp}. Let $e: X \rightarrow Y$ and $e': Y \rightarrow Z$ be coarsely monotone maps, thought of as morphisms in the coarse category, and factorize $e'e$ as $e'e = me''$ with $m$ coarsely light and $e''$ coarsely monotone. Then by (3) we have the morphism $h$ in the following commutative diagram:
$$ \xymatrix{
\bullet \ar[r]^e \ar@{=}[d] & \bullet \ar@{..>}[d]^h \ar@{=}[r] & \bullet \ar[d]^{e'} \\
\bullet \ar[r]_{e''} & \bullet \ar[r]_{m} & \bullet \\
}
$$
We also have the morphism $j$ in the diagram
$$ \xymatrix{
\bullet \ar[r]^{e'} \ar@{=}[d] & \bullet \ar@{..>}[d]^j \ar@{=}[r] & \bullet \ar@{=}[d] \\
\bullet \ar[r]_{h'} & \bullet \ar[r]_{mi} & \bullet \\
}
$$
where $h = ih'$ is the $(\mathsf{CMon}, \mathsf{CLight})$-factorization of $h$. Thus $ij$ is a right inverse (in the coarse category) to $m$. To show that it is a two-sided inverse, we apply the uniqueness part of (3) to the commutative diagram
$$ \xymatrix{
\bullet \ar[r]^{e''} \ar@{=}[d] & \bullet \ar@{..>}[d]^{ijm} \ar[r]^m & \bullet \ar@{=}[d] \\
\bullet \ar[r]_{e''} & \bullet \ar[r]_{m} & \bullet \\
}
$$
\end{proof}

\begin{Remark}
In the language of category theory, Lemma~\ref{rightM} states that the coarse category admits a \textbf{right $\mathsf{CLight}$-factorization system} in the sense of~\cite{DikTho} (see also~\cite{EhrWyl}). It is well known that for a category $\mathbb{C}$ and a class $\mathcal{M}$ of morphisms in $\mathbb{C}$, if $\mathcal{M}$ contains the isomorphisms and is closed under composition and $\mathbb{C}$ admits a right $\mathcal{M}$-factorization system, then $\mathcal{M}$ is part of a unique factorization system $(\mathcal{E}, \mathcal{M})$ on the category. Thus the above theorem can be proved using categorical arguments once we have Lemmas~\ref{comp} and \ref{rightM}.
\end{Remark}

We will call a factorization $f \sim f'e$ with $f'$ coarsely light and $e$ coarsely monotone a \textbf{coarse monotone-light factorization of $f$}. Note that, by condition (3) in the definition of a factorization system, such a factorization is unique up to a coarse equivalence which makes the obvious diagram commute up to closeness. A large number of useful properties of coarsely monotone and light maps follow from the above theorem and general facts about factorization systems. For example, if $fg$ and $f$ are coarsely light, then so is $g$; dually, if $fg$ and $g$ are coarsely monotone, then so is $f$. We conclude this section with the following easy observations, which show how coarsely light and coarsely monotone maps can be used to characterise certain types of ls-spaces.

\begin{Proposition} \label{mapsto1}
An ls-space $X$ has asymptotic dimension $0$ if and only if every ls-continuous map $f: X \rightarrow Y$ is coarsely light. An ls-space $X$ is $\mathcal{U}$-connected for some uniformly bounded cover $\mathcal{U}$ if and only if every ls-continuous map $f: X \rightarrow B$ to a bounded space $B$ is coarsely monotone.
\end{Proposition}

\begin{Remark}\label{RoeCoarse}
If $f$, $g$ and $h$ are ls-continuous maps such that $f = gh$ and $h$ is surjective, then $g$ and $h$ are both proper whenever $f$ is proper. From this and Theorem \ref{factsys} it follows easily that the classes of coarsely monotone coarse maps and coarsely light coarse maps form a factorization system on Roe's coarse category (i.e.~the subcategory of the coarse category consisting of the coarse maps). 
\end{Remark}

\section{Pullbacks in the coarse category} \label{secPb}
In this section we collect some basic facts about pullbacks in the coarse category which we will need in the next section. Suppose the following diagram of ls-continuous maps represents a pullback in the coarse category:

\begin{equation} \label{pb1}
\begin{gathered}
\xymatrix {
P \ar[d]_g \ar[r]^{j} & C \ar[d]^f \\
A \ar[r]_h & B \\
}
\end{gathered}
\end{equation}
Explicitly, this means that the above square commutes up to closeness, and for any ls-space $X$ and any ls-continuous maps $u: X \rightarrow A$ and $v: X \rightarrow C$ such that $hu \sim fv$, there is a unique-up-to-closeness map $w: X \rightarrow P$ such that $gw \sim u$ and $jw \sim v$. Consider the product $A \times C$, with the ls-structure given by all families $\mathcal{U}$ such that $\pi_A(\mathcal{U})$ and $\pi_C(\mathcal{U})$ are uniformly bounded, where $\pi_A: A \times C \rightarrow A$ and $\pi_C: A \times C \rightarrow C$ are the evident projections. There is a canonical ls-continuous map $k = (g,j): P \rightarrow A \times C$.

\begin{Lemma} \label{emb}
The map $k = (g,j)$ given above is a coarse embedding.
\end{Lemma}
\begin{proof}
Suppose not. Let $\mathcal{U}$ be a uniformly bounded family in $A \times C$ such that $k^{-1}(\mathcal{U})$ is not uniformly bounded. Let $W$ be the set $\{(p, p') \in P \times P \mid k(p)\mathcal{U}k(p') \}$ equipped with the ls-structure consisting only of families of singleton sets. The projections $\alpha_1, \alpha_2: W \rightarrow P$ are ls-continuous, and by construction, $k\alpha_1$ and $k \alpha_2$ are close. It follows that $g \alpha_1 \sim g \alpha_2$ and $j \alpha_1 \sim j \alpha_2$. But then the uniqueness part of the pullback property forces $\alpha_1$ and $\alpha_2$ to be close, which cannot be the case since $k^{-1}(\mathcal{U})$ is not uniformly bounded.
\end{proof}

\begin{Remark}
The above lemma also follows from the observation that the monomorphisms in the coarse category are precisely the coarse embeddings. 
\end{Remark}

From Lemma \ref{emb} above it follows that the pullback, when it exists, can be canonically embedded into the product. Thus, we can replace $P$ by the image of the map $k$ and replace $j$ and $g$ by the obvious projections and still obtain a pullback. The fact that the resulting diagram must commute up to closeness means that we may always assume that $P$ is a subset of 
$$
A \times_\mathcal{S} C = \{(a, c) \in A \times C \mid h(a)\mathcal{S}f(c)\}
$$
for some $\mathcal{S}$ a uniformly bounded family in $B$. We are now ready to present an example to show that not all pullbacks exist in the coarse category.

\begin{Example} \label{nopb}
Let $A$ be the one point set, $B$ be the natural numbers with the ls-structure arising from the usual metric, and $C$ the subspace 
$$ \{(a_0, a_1 \ldots ) \mid (i > a_0) \Rightarrow (a_i = 0) \} \subseteq \bigoplus_{i=1}^{\infty} \mathbb{N}$$
where the metric on $\bigoplus_{i=1}^{\infty} \mathbb{N}$ is Euclidean distance. Let $h: A \rightarrow B$ be the inclusion of $0$ and let $f: C \rightarrow B$ be the projection $(a_0, a_1, \ldots) \mapsto a_0$. Then $f$ and $h$ are ls-continuous, but no pullback of $f$ along $h$ exists.
\end{Example}
\begin{proof}
Suppose a pullback does exist, given by the diagram (\ref{pb1}). By the above arguments, we may set $P \subseteq A \times_\mathcal{S} C$ for some uniformly bounded family $\mathcal{S}$ in $B$. Since $h$ is the inclusion of $0$, this means that $P$ can be viewed as a subspace of $f^{-1}(B(0, N))$ for some $N > 0$. Let $Q$ be the subspace $f^{-1}(B(0, N+1))$, with $l: Q \rightarrow C$ the obvious inclusion and $m$ the unique map to $A$. Since $hm \sim fl$, by the property of the pullback, there must be a unique ls-continuous map $k:  Q \rightarrow P$ such that $jk \sim l$. In particular, everything in the image of $l$ should be a bounded distance from the image of $j$, but while the $(N+1)^{\mathrm{th}}$ coordinate in $P$ has to be $0$, in $Q$ it can be arbitrary, giving the required contradiction. 
\end{proof}

Note that products do exist in the coarse category (as defined in the present paper): the product of a ls-space $A$ and $B$ is given by the set $A \times B$ with the ls-structure consisting of all those families $\mathcal{U}$ such that $\pi_A(\mathcal{U})$ and $\pi_B(\mathcal{U})$ are both uniformly bounded, where $\pi_A$ and $\pi_B$ are the projections.

\section{Coarsely monotone maps arising from a reflection}
Recall that a subcategory $\mathbb{X}$ of a category $\mathbb{C}$ is \textbf{reflective} if for every object $C$ of $\mathbb{C}$, there is an object $I(C)$ of $\mathbb{X}$ and morphism $\eta_C: C \rightarrow I(C)$ which is universal in the following sense: for any $X$ an object of $\mathbb{X}$ and $f: C \rightarrow X$ a morphism in $\mathbb{C}$, there is a unique morphism $g: I(C) \rightarrow X$ in $\mathbb{X}$ such that $g\eta_C = f$. The assignment $C \mapsto I(C)$ extends to a functor $I: \mathbb{C} \rightarrow \mathbb{X}$ in the obvious way, which we call the \textbf{reflection} of $\mathbb{C}$ onto $\mathbb{X}$. 

Let $\mathbf{CHaus}$ be the category of compact Hausdorff spaces and continuous maps (the category on which the classical monotone and lights maps constitute a factorization system), and let $\mathbf{CHaus}_0$ be the full subcategory consisting of the totally disconnected spaces. Then $\mathbf{CHaus}_0$ is reflective in $\mathbf{CHaus}$. Indeed, for $X$ a compact Hausdorff space, let $I(X)$ be the set of  connected components of $X$ with the quotient topology, and let $\eta_X: X \rightarrow I(X)$ be the quotient map. Then $I(X)$ is a totally disconnected compact Hausdorff space, and $\eta_X$ is the universal continuous map from $X$ to a totally disconnected compact Hausdorff space (see for example \cite{Bourb}). 

The monotone maps between compact Hausdorff spaces can be recovered from this reflection in the following way. For a general reflection $I: \mathbb{C} \rightarrow \mathbb{X}$, let $\mathcal{E}_I$ be the class of morphisms $f$ such that $I(f)$ is an isomorphism. Thus in the case currently being considered, $\mathcal{E}_I$ is the class of continuous maps which induce a bijection on connected components. Every monotone map is in $\mathcal{E}_I$, but not every map in $\mathcal{E}_I$ is monotone. It is easy to check that a continuous map $f$ is monotone if and only if for every pullback
$$ \xymatrix{
P \ar[d]_g \ar[r] & X \ar[d]^f \\
1 \ar[r] & Y \\
}
$$
where $1$ is the one point space, the map $g$ is in $\mathcal{E}_I$. 

\begin{Remark}
In fact, a continuous map $f$ is monotone if and only if \emph{every} pullback of $f$ is in $\mathcal{E}_I$. The construction of the class of monotone maps from the reflection $I$ is a special case of a much more general process outlined in \cite{CarJanKelPar}. The context in \cite{CarJanKelPar}, however, is that of a category admitting pullbacks, of which the coarse category is not an example as we have seen.
\end{Remark}

Since a compact Hausdorff space is totally disconnected if and only if it has inductive dimension $0$ (see for example \cite{Arh}), one might wonder if a similar process can be applied using the class of ls-spaces of asymptotic dimension zero to arrive at the class of coarsely monotone maps. We must first describe the reflection. If $X$ is an ls-space, then let $I(X)$ be the ls-space whose underlying set is the same as $X$, but whose ls-structure consists of all families which refine the set of $\mathcal{U}$-components of $X$ for some uniformly bounded family $\mathcal{U}$ in $X$. Clearly $I(X)$ has asymptotic dimension zero, and the identity set map $\eta_X: X \rightarrow I(X)$ is ls-continuous. The following lemma is easy to prove.

\begin{Lemma}
Let $f: X \rightarrow Y$ be an ls-continuous map, where $\asdim Y = 0$. Then $f$ factors (up to closeness) uniquely (up to closeness) through $\eta_X: X \rightarrow I(X)$.
\end{Lemma}

It follows that the assignment $X \mapsto I(X)$ gives a reflection from the coarse category to the full subcategory of ls-spaces of asymptotic dimension zero. In particular, the assignment $X \mapsto I(X)$ extends to a functor $I$, which (in terms of representatives) assigns to each ls-continuous map $f: X \rightarrow Y$ an ls-continuous map $I(f): I(X) \rightarrow I(Y)$ (which is the same as $f$ at the level of underlying sets). As in the classical case, we let $\mathcal{E}_I$ be the class of all (equivalence classes of) ls-continuous maps $f$ such that $I(f)$ is a isomorphism (i.e.~is represented by a coarse equivalence). Clearly:

\begin{Lemma}\label{eichar}
A coarsely surjective ls-continuous map $f: X \rightarrow Y$ is in $\mathcal{E}_I$ if and only if for every $\mathcal{V}$ a uniformly bounded cover in $Y$ there is a uniformly bounded family $\mathcal{U}$ in $X$ such that the inverse image of each $\mathcal{V}$-component of $Y$ is contained in a $\mathcal{U}$-component of $X$.
\end{Lemma}

We will need the following easy categorical lemma.

\begin{Lemma}\label{rightcancel}
Let $I: \mathbb{C} \rightarrow \mathbb{X}$ be any reflection. Then, for any pair of morphisms $f: A \rightarrow B$, $g: B \rightarrow C$, such that $f$ has a right inverse, we have 
$$
gf \in \mathcal{E}_I \Rightarrow g \in \mathcal{E}_I.
$$
\end{Lemma}
\begin{proof}
Let $h$ be the two-sided inverse of $I(gf) = I(g)I(f)$. Then $I(f)h$ is a right inverse of $I(g)$. We also have 
$$
I(f)hI(g)I(f) = I(f)
$$ 
so applying the right inverse of $I(f)$ to both sides, we have that $I(f)h$ is a two-sided inverse of $I(g)$ as required.
\end{proof}

We can already observe from Lemma \ref{eichar} that coarsely monotone maps are always in $\mathcal{E}_I$. The converse, however, does not hold. By analogy with the classical case, one might ask for a characterisation of coarsely monotone maps in terms of stability under certain pullbacks. However, as we have seen, not all pullbacks exist in the coarse category. This motivates us to instead consider an alternative condition.

\begin{Lemma}\label{equipb}
Let $\mathbb{C}$ be a category which admits all pullbacks, and let $\mathcal{E}_I$ be the class of morphisms inverted by a reflection $I: \mathbb{C} \rightarrow \mathbb{X}$. Then the following are equivalent for morphisms $f: X \rightarrow Y$ and $j: Z \rightarrow Y$ in $\mathbb{C}$:
\begin{itemize}
\item[$(\mathsf{P_1})$] the pullback of $f$ along $j$ is in $\mathcal{E}_I$,
\item[$(\mathsf{P_2})$] for every commutative diagram 
$$ \xymatrix{
W \ar[r]^i \ar[d]_g & X \ar[d]^f \\
Z \ar[r]_j & Y\\
}
$$
there is a commutative diagram with $e \in \mathcal{E}_I$
\begin{equation}\label{factor}
\begin{gathered}
 \xymatrix{
W \ar[rd]^d \ar@/_/[rdd]_g \ar@/^/[rrd]^i \\
& P \ar[d]_e \ar[r]^c & X \ar[d]^f \\
& Z \ar[r]_j & Y\\
}
\end{gathered}
\end{equation}
\end{itemize}
\end{Lemma}
\begin{proof}
$(\mathsf{P_1} \Rightarrow (\mathsf{P_2})$ is obvious -- simply let $e$ be the pullback of $f$ along $j$.

$(\mathsf{P_2} \Rightarrow (\mathsf{P_1})$: In (\ref{factor}) above, let $W$ be the pullback of $f$ along $j$ and $i$ and $g$ the projections. It follows from the property of the pullback that $d$ has a left inverse $s$ such that $gs = e \in \mathcal{E}_I$. Thus $s$ has a right inverse, so applying Lemma~\ref{rightcancel}, we obtain $g \in \mathcal{E}_I$ as required.
\end{proof}

Let $\mathcal{F}$ be a class of morphisms in a category $\mathbb{C}$. We say that a morphism $f: X \rightarrow Y$ in $\mathcal{E}_I$ is \textbf{stably in $\mathcal{E}_I$ with respect to $\mathcal{F}$} if for every $j: Z \rightarrow Y$ in $\mathcal{F}$, the condition $(\mathsf{P}_2)$ in the above lemma is satisfied. The classical monotone maps are thus precisely the continuous maps which are stably in $\mathcal{E}_I$ with respect to all maps $j$ whose domain is the one point space (where $I$ is the reflection onto the totally disconnected spaces). We are now ready to state an analogue for coarsely monotone maps. Recall that a ls-space is called \emph{monogenic} if its ls-structure is generated by a single family of subsets \cite{Roe}.

\begin{Theorem} \label{stabilize}
Let $\mathcal{F}$ be the set of all (equivalence classes of) ls-continuous maps whose domain is monogenic. Then a map $f: X \rightarrow Y$ is coarsely monotone if and only if it is stably in $\mathcal{E}_I$ with respect to $\mathcal{F}$, where $I$ is the reflection onto ls-spaces of asymptotic dimension $0$.
\end{Theorem}
\begin{proof}
$(\Rightarrow)$ Consider a diagram
$$ \xymatrix{
W \ar[r]^i \ar[d]_g & X \ar[d]^f \\
Z \ar[r]_j & Y\\
}
$$
of ls-continuous maps which commutes up to closeness as witnessed by the uniformly bounded cover $\mathcal{T}$ in $Y$, where $f$ is coarsely monotone and $\mathcal{U}$ is a cover that generates the ls-structure on $Z$. Since $f$ is coarsely surjective, there is some $\mathcal{R}$ a uniformly bounded cover in $Y$ such that $Y \subseteq \st(\mathsf{Im}(f), \mathcal{R})$. We now construct the $P$ in diagram (\ref{factor}). For every $y \in Y$, select an $s(y) \in X$ such that $y\mathcal{R}fs(y)$. Define the uniformly bounded family 
$$
\mathcal{S} = \st(\st(j(\mathcal{U}), \mathcal{T}), \mathcal{R}) .
$$
Since $f$ is coarsely monotone, there is a uniformly bounded family $\mathcal{W}$ in $X$ and a family $\mathcal{Q}$ of $\mathcal{W}$-conncted subsets in $X$ such that $f(\mathcal{Q})$ is uniformly bounded. Define $P$ to be the subspace of $Z \times X$ consisting of all pairs $(z, x)$ such that $js(z)\mathcal{Q}x$, and let $e: P \rightarrow Z$ and $c: P \rightarrow X$ be the projections. Note that $e$ is surjective, and that $je$ and $fc$ are close as witnessed by $\st(f(\mathcal{Q}), \mathcal{R})$. If $f(x) \mathcal{T} j(z)$, then by construction of $\mathcal{Q}$, $(z, x)$ is in $P$, so there is a canonical map $d: W \rightarrow P$ making diagram (\ref{factor}) commute. It remains to show that $e$ is in $\mathcal{E}_I$. Suppose $(z, x)$ and $(z', x')$ are points in $P$ such that $z \mathcal{U} z'$. We claim that $(z, x)$ and $(z', x')$ are $\mathcal{U} \times \mathcal{W}$-connected in $P$. Indeed, $(z, x)$ is clearly $\Delta \times \mathcal{W}$-connected to $(z, sj(z))$ in $P$ (where $\Delta$ is the trivial cover by singletons), and similarly, $(z', x')$ is $\Delta \times \mathcal{W}$-connected to $(z', sj(z'))$ in $P$. Since $fsj(z)$ and $fsj(z')$ are both in some element of $\mathcal{S}$, there is a $\Delta \times \mathcal{W}$ chain from $(z, sj(z))$ to $(z, sj(z'))$. Finally, $(z, sj(z'))$ and $(z', sj(z'))$ are $\mathcal{U} \times \Delta$-connected. Composing the chains, we obtain a proof of the claim. This shows that the inverse image of a $\mathcal{U}$-component of $Z$ is $\mathcal{U} \times \mathcal{W}$-connected. Since the ls-structure on $Z$ is generated by $\mathcal{U}$, this is enough to show that $e$ is in $\mathcal{E}_I$. 

$(\Leftarrow)$ We first claim that $f$ is coarsely surjective. In the diagram
$$ \xymatrix{
W \ar[r]^i \ar[d]_g & X \ar[d]^f \\
Z \ar[r]_j & Y\\
}
$$
let $Z$ be the underlying set of $Y$ equipped with the smallest ls-structure, i.e.~such that the only uniformly bounded families are families of singletons. Let $j$ be the identity set map and let $W$ be the empty ls-space with the empty maps $g$ and $i$. Then there is a diagram (\ref{factor}) with $e \in \mathcal{E}_I$. In particular, $e$ must be surjective, and it follows by commutativity of (\ref{factor}) that $f$ must be coarsely surjective. Now let $\mathcal{U}$ be a uniformly bounded cover of $Y$. For each element $U$ of $\mathcal{U}$, pick a point $s_U \in U$. Let $Z$ be the set of all such $s_U$ equipped with the smallest ls-structure, and let $j$ be the map induced by the inclusion of each point $s_U$. Clearly $Z$ is monogenic as an ls-space. Let $W$ be the set of all pairs $(x, U)$ such that $U \in \mathcal{U}$ and $x \in f^{-1}(U)$, let $i: W \rightarrow X$ be the map $(x, U) \mapsto x$ and $g: W \rightarrow Z$ the map $(x, U) \rightarrow s_U$. Put the largest ls-structure on $W$ for which $i$ and $g$ are ls-continuous. In particular, $\{(x, U), (y, V)\}$ is never bounded if $U \neq V$. Since $fi$ and $jg$ commute up to closeness (as witnessed by $\mathcal{U}$), we must have a diagram of the form of  (\ref{factor}) with $e \in \mathcal{E}_I$. Since $e$ is in $\mathcal{E}_I$, each $e^{-1}(s_U)$ must be $\mathcal{W}$-connected for some fixed uniformly bounded family $\mathcal{W}$ in $P$. Thus for every $U \in \mathcal{U}$, $f^{-1}(U) = i(g^{-1}(s_U))$ must be contained in $\st(c(e^{-1}(s_U)), \mathcal{T})$, a $\st(c(\mathcal{W}), \mathcal{T})$-connected subset of $X$, where $\mathcal{T}$ witnesses the closeness of $cd$ and $i$. Moreover, the family $f(\st(c(e^{-1}(s_U)), \mathcal{T}))$ is uniformly bounded because $je \sim fc$. Since $\mathcal{U}$ was arbitrary, $f$ is coarsely monotone as required.
\end{proof}

Recall from \cite{Roe} that an ls-space is monogenic if and only if it is coarsely equivalent to a geodesic $\infty$-metric space (that is, in which points which are finite distance apart are connected by a geodesic). Thus we have the following corollary of the above result.

\begin{Corollary}
Let $\mathcal{F}$ be the set of all (equivalence classes of) ls-continuous maps whose domain is a geodesic $\infty$-metric space. Then a map $f: X \rightarrow Y$ is coarsely monotone if and only if it is stably in $\mathcal{E}_I$ with respect to $\mathcal{F}$, where $I$ is the reflection onto ls-spaces of asymptotic dimension $0$.
\end{Corollary}

A result which more closely resembles the topological situation (which involves pullbacks along maps from the singleton space) is as follows, where the singleton set is replaced by a disjoint union of singleton ls-spaces. 

\begin{Corollary}
Let $\mathcal{F}$ be the set of all (equivalence classes of) ls-continuous maps whose domain is a set with the trivial ls-structure (i.e.~the ls-structure consisting of families of singleton sets). Then a map $f: X \rightarrow Y$ is coarsely monotone if and only if it is stably in $\mathcal{E}_I$ with respect to $\mathcal{F}$, where $I$ is the reflection onto ls-spaces of asymptotic dimension $0$. 
\end{Corollary}
\begin{proof}
$(\Rightarrow)$: This follows from Theorem \ref{stabilize} and the fact that the trivial ls-structure is monogenic.

$(\Leftarrow)$: This follows from the proof of Theorem \ref{stabilize}.
\end{proof}

\begin{Remark}
One can easily show that a continuous map $f: A \rightarrow B$ between compact Hausdorff spaces is (classically) monotone if and only if for every open set $V \in B$, the restriction $f|_{f^{-1}(V)}$ induces a bijection between connected components. We can thus think of classical monotone maps as those that induce bijections on connected components ``on any fixed small scale'', i.e.~open neighbourhood. We can view Theorem \ref{stabilize} as saying that coarsely monotone maps are those that induce bijections between $\mathcal{U}$-components ``on any fixed large scale''.
\end{Remark}

\begin{Remark}
The absence of pullbacks in the coarse category presents significant obstacles to applying basic category theory arguments in coarse geometry. The above result shows that it can sometimes be useful to consider weaker notions along the lines of condition $(\mathsf{P_2})$ in the coarse category, which in categories that do admit pullbacks reduce to familiar notions.
\end{Remark}

\section{Pullback-stability of coarsely monotone maps}
Given any factorization system $(\mathcal{E}, \mathcal{M})$ on a category $\mathbb{C}$, it is always the case that the class $\mathcal{M}$ is stable under all pullbacks that exist in the category (that is, every pullback of an element of $\mathcal{M}$ is in $\mathcal{M}$). The same is not true in general for the class $\mathcal{E}$. Many important factorization systems do have this property, though, including the classical monotone-light factorization system. We now show that this also holds for the coarse monotone-light factorization system.

\begin{Proposition}
If the diagram below is a pullback in the coarse category and $f$ is coarsely monotone, then $g$ is also monotone.
$$
\xymatrix {
P \ar[d]_g \ar[r]^{j} & C \ar[d]^f \\
A \ar[r]_h & B \\
}
$$
\end{Proposition}
\begin{proof}
By the remarks in Section \ref{secPb}, we may assume that $P$ is a subspace of $A \times C$ with $g$ and $j$ the projections. Let $\mathcal{P}$ be the uniformly bounded family in $B$ which witnesses the closeness of $fj$ and $hg$, extended to a cover. We first show that $g$ is coarsely surjective. We know that $f$ is coarsely surjective, so suppose that $B \subseteq \st(\mathsf{Im}(f), \mathcal{R})$ for some uniformly bounded cover $\mathcal{R}$ of $B$. Consider the subspace of $A \times C$ 
$$
R = \{(a, c) \mid a \in A,\ c \in C,\ h(a)\mathcal{R}f(c) \}
$$
and the obvious projections $\rho_1$ and $\rho_2$ to $A$ and $C$ respectively. Note that $\rho_1$ is surjective. By the property of the pullback, there is a map $l: R \rightarrow P$ such that $gl \sim \rho_1$. Since $\rho_1$ is surjective, it follows that $gl$ is coarsely surjective, and consequently that $g$ is coarsely surjective as required. 

Now, let $\mathcal{V}$ be a uniformly bounded family in $A$, and consider the uniformly bounded family $h(\mathcal{V})$ in $B$. Let $\mathcal{V}' = \st(h(\mathcal{V}), \mathcal{P})$. Since $f$ is coarsely monotone, there is a uniformly bounded cover $\mathcal{V}''$ of $B$, which we may suppose to be coarser than $\mathcal{V}'$, and a uniformly bounded cover $\mathcal{W}$ in $C$ such that each element of $f^{-1}(\mathcal{V}')$ is contained in a $\mathcal{W}$-component of an element of $f^{-1}(\mathcal{V}'')$. Let $Q$ be the subspace
$$
\{(a, c) \mid a \in A, c \in C, h(a)\mathcal{V}''f(c) \}
$$
and let $\pi_1: Q \rightarrow A$, $\pi_2: Q \rightarrow C$ be the obvious projections. Since $h \pi_1$ is close to $f\pi_2$, by the property of the pullback, there must be an ls-continuous map $k: Q \rightarrow P$ such that $gk \sim \pi_1$ and $jk \sim \pi_2$. Since a map which is close to an ls-continuous map is ls-continuous, we may suppose that $k$ is the identity on those elements which are also in $P$. Consider the uniformly bounded cover 

$$\mathcal{V} \times \mathcal{W} = \{V \times W \mid V \in \mathcal{V},\ W \in \mathcal{W} \}$$
of $A \times C$ restricted to $Q$, and let $\mathcal{T}$ be its image under $k$. We claim that $g^{-1}(\mathcal{V})$ refines the set of $\mathcal{T}$-components of $g^{-1}(\st(\mathcal{V}, \mathcal{X}))$, where $\mathcal{X}$ is the cover of $A$ which witnesses the closeness of $kg$ and $\pi_1$. Indeed, suppose $(a,c)$ and $(a',c')$ are elements of $P$ such that $\{a, a'\} \subseteq V \in \mathcal{V}$. Then there is a $V' \in \mathcal{V}'$ containing all of $f(c)$, $f(c')$, $h(a)$ and $h(a')$, and a $V'' \in \mathcal{V}''$ containing $V'$ such that $c$ and $c'$ are $\mathcal{W}$-connected inside $f^{-1}(V'')$. Consider the chain
$$
\xymatrix{
(a, c) \ar@{-}[r] & (a', c) \ar@{-}[r] & (a', c_1) \ar@{-}[r] & \cdots \ar@{-}[r] & (a', c') \\
}
$$
where $c, c_1, \ldots, c'$ is a chain of $\mathcal{W}$-related elements in $f^{-1}(V'')$. First note that every element in this chain is in $Q$, and in particular, in $\pi_1^{-1}(V) \subseteq Q$. Moreover, each pair of consecutive elements is related by $\mathcal{V} \times \mathcal{W}$. Taking the image of the chain under $k$, it follows that $(a,c)$ and $(a', c')$ are connected by a $\mathcal{T}$-chain inside an element of $g^{-1}(\st(\mathcal{V}, \mathcal{X}))$. 
\end{proof}

\section{Maps extended to the Higson corona} \label{secCorona}
We will now deal with the special case of a map $f: X \rightarrow Y$ where $X$ and $Y$ are proper metric spaces (with the induced ls-structures) and $f$ is a coarse map (i.e.~an ls-continuous map such that the inverse image of a bounded set in $Y$ is bounded in $X$). Such a map induces a continuous map $\nu f: \nu X \rightarrow \nu Y$ between the \textbf{Higson coronas} of $X$ and $Y$. We briefly recall the relevant definitions from \cite{Roe}.

\begin{Definition}
Let $X$ be a metric space, and $g: X \rightarrow \mathbb{C}$ a bounded function to the complex numbers. Then $g$ is said to be \textbf{slowly oscillating} if for every $R > 0$ and $\varepsilon > 0$ there is a bounded set $B$ in $X$ such that $d(x, x') \leq R \Rightarrow |g(x) - g(x')| \leq \varepsilon$ for $x,x' \in X \setminus B$.
\end{Definition}

The \textbf{Higson compactification} $hX$ of a proper metric space $X$ is the compactification of $X$ characterised by the fact that a bounded continuous complex-valued function $g: X \rightarrow \mathbb{C}$ extends continuously to $\tilde{g}: hX \rightarrow \mathbb{C}$ if and only if it is slowly oscillating. In particular, $X$ is dense in $hX$. The complement $\nu X = hX \setminus X$ is called the \textbf{Higson corona} of $X$. 

In terms of $C^\ast$ algebras, $hX$ is the compact Hausdorff space which corresponds (under Gelfand duality) to the algebra $C_h(X)$, while $\nu X$ corresponds to $C_h(X)/C_0(X)$, where $C_h(X)$ is the algebra of continuous bounded slowly-oscillating complex-valued functions on $X$ and $C_0(X)$ is the ideal of $C_h(X)$ consisting of those functions which tend to zero at infinity. Let $B_h(X)$ be the algebra of (not necessarily continuous) bounded slowly oscillating complex-valued functions on $X$ and $B_0(X)$ be the ideal of $B_h(X)$ consisting of those functions which tend to zero at infinity. Recall from \cite{Roe} that $C_h(X)/C_0(X)$ is canonically isomorphic to $B_h(X)/B_0(X)$ via the map induced by the inclusion $C_h(X) \rightarrow B_h(X)$. Thus the corona can be defined in a way which is independent of the topology of $X$, which is expected since the corona captures large-scale behaviour of $X$. 

Given a coarse map $f: X \rightarrow Y$ between metric spaces, the map $[g] \mapsto [gf]$ defines a $\ast$-homomorphism $f^\ast: B_h(Y)/B_0(Y) \rightarrow B_h(X)/B_0(X)$, which corresponds under Gelfand duality to a continuous map $\nu f: \nu X \rightarrow \nu Y$. This defines a functor $\nu$ from the category of proper metric spaces and closeness classes of coarse maps to the category of compact Hausdorff spaces~\cite{Roe}. In particular, $\nu f$ is a homeomorphism whenever $f$ is a coarse equivalence.

\begin{Lemma}
Let $f: X \rightarrow Y$ be a coarse map between proper metric spaces and $\nu f: \nu X \rightarrow \nu Y$ the induced map. If $f$ is continuous, then $\nu f$ is (up to homeomorphism) the restriction to $\nu X$ of the unique continuous extension $hf: hX \rightarrow hY$ of $f$. 
\end{Lemma}
\begin{proof}
One only needs to check that the diagram 
$$ \xymatrix{
C_h(Y)/C_0(Y) \ar[d]_{\cong} \ar[rr]^{[g] \mapsto [gf]} & & C_h(X)/C_0(X) \ar[d]^{\cong} \\
B_h(Y)/B_0(Y) \ar[rr]^{ f^\ast} & & B_h(X)/B_0(X) \\
}
$$
commutes, which is easy.
\end{proof}

We will need the following lemma, which is taken from \cite{DM}.

\begin{Lemma} \label{ext}
Any slowly oscillating function from a subset $A$ of a proper metric space $X$ to $[0,1]$ extends to a slowly oscillating function on the whole of $X$ to $[0,1]$.
\end{Lemma}

\begin{Corollary} \label{extcor}
Any bounded slowly oscillating function $f$ from a subset $A$ of a proper metric space $X$ to $\mathbb{C}$ extends to a bounded slowly oscillating function on the whole of $X$ to $\mathbb{C}$.
\end{Corollary}
\begin{proof}
Rescaling by a constant and translating preserves slowly oscillating functions, so we may assume that $f$ has image $[0,1] \times [0,1]$. The projections $\pi_1 f$ and $\pi_2 f$ are slowly oscillating functions from $A$ to $[0,1]$, so by Lemma \ref{ext} we can extend each of them to slowly oscillating functions on $X$. Taking the induced map from $X$ to $[0,1] \times [0,1]$ we obtain the required extension.
\end{proof}

\begin{Proposition} \label{coronaeq}
Let $f: X \rightarrow Y$ be an ls-continuous map and $\nu f: \nu X \rightarrow \nu Y$ the induced continuous map between Higson coronas. Then
\begin{itemize}
\item[(1)] $\nu f$ is injective if and only if $f$ is a coarse embedding;
\item[(2)] $\nu f$ is surjective if and only if $f$ is coarsely surjective;
\item[(3)] $\nu f$ is a homeomorphism if and only if $f$ is a coarse equivalence.
\end{itemize}
\end{Proposition}
\begin{proof}
(1): Suppose $\nu f$ is injective; under duality, this is the same as to say that $f^\ast: B_h(Y)/B_0(Y) \rightarrow B_h(X)/B_0(X)$ is surjective. Suppose that $f$ is not a coarse embedding. Pick a sequence of points $(a_n, b_n)$ in $X$ such that $d(a_n, b_n)$ tends to infinity, but $d(f(a_n), f(b_n))$ is bounded. Since $f$ is proper, $a_n$ and $b_n$ cannot be bounded, so we may choose the $a_n$ and $b_n$ such that each $a_i$ (resp. $b_i$) is at least $i$ away from all the $b_j$ (resp. $a_j$) for $j < i$. Define a map on the union of the $a_n$ and the $b_n$ to $[0,1]$ which sends every $a_n$ to $0$ and every $b_n$ to $1$. This map is slowly oscillating, so we can extend it to a slowly-oscillating map on the whole of $X$. However, this map cannot be written as $gf + b$ for $g \in B_h(Y)$ and $b \in B_0(X)$, which contradicts the surjectivity of $f^\ast: B_h(Y)/B_0(Y) \rightarrow B_h(X)/B_0(X)$. Now suppose that $f$ is a coarse embedding. Then $f$ is, up to coarse equivalence, the inclusion of its image into $Y$. By Lemma \ref{extcor}, such inclusions give rise to surjective maps $ B_h(Y)/B_0(Y) \rightarrow B_h(X)/B_0(X)$, which gives the required result.

(2): Suppose $\nu f$ is surjective, i.e.~that $f^\ast: B_h(Y)/B_0(Y) \rightarrow B_h(X)/B_0(X)$ is injective, and that $f$ is not coarsely surjective. For every $n$, pick a point $y_n$ in $Y$ such that $d(f(X), Y) \geq n$. Define a function on $\{y_n\}$ to $\mathbb{C}$ which sends every $y_n$ to $1$. This map is clearly slowly oscillating, so it extends by Lemma \ref{ext} to a slowly oscillating function $g: Y \rightarrow \mathbb{C}$. Then $g$ is not in $B_0(Y)$ but $gf$ is in $B_0(X)$, contradicting the fact that $f^\ast: B_h(Y)/B_0(Y) \rightarrow B_h(X)/B_0(X)$ is injective. The other direction is to say that if $f$ is coarsely surjective, and $gf$ is in $B_0(X)$, then $g$ is in $B_0(Y)$, which is easy to check.

(3) is a consequence of (1) and (2).
\end{proof}

If we are given a coarse map $f: X \rightarrow Y$ between proper metric spaces, then we can consider a $1$-net $X_1$ in $X$ (i.e.~a maximal $1$-separated subset of $X$). The space $X_1$ is a proper metric space, the inclusion $i: X_1 \rightarrow X$ is a coarse equivalence, and the composite $fi$ is a continuous coarse map. As a result, the induced map $\nu (fi): X_1 \rightarrow Y$ is the unique extension of $fi$ to $hX_1 \rightarrow hY$ restricted to $\nu X_1$. But $\nu i$ is a homeomorphism, so $\nu f$ and $\nu(fi)$ are the same up to homeomorphism. Thus in the rest of the section we will often assume that a given coarse map $f$ is actually continuous, and that the map $\nu f$ is its extension restricted to the corona. The following lemma is a special case of Proposition 2.3 from \cite{Dranetal}.

\begin{Lemma} \label{int}
For $A$ and $B$ two subsets of a proper metric space $X$, $\overline{A} \cap \overline{B} \cap \nu X$ is non-empty if and only if there is a $S > 0$ such that $B(A, S) \cap B(B, S)$ is unbounded.
\end{Lemma}

The proof of the following theorem was inspired in part by the techniques used in \cite{AustinVirk}.

\begin{Theorem} \label{corona}
Let $f: X \rightarrow Y$ be a coarse map between proper metric spaces. Then $f$ is coarsely monotone if and only if the induced map $\nu f: \nu X \rightarrow \nu Y$ is (classically) monotone.
\end{Theorem}
\begin{proof}
By the above remarks, we may assume that $X$ is topologically discrete, and that $\nu f$ is the restriction of a continuous extension $hf$. 

$(\Rightarrow)$ By Proposition \ref{coronaeq}, $\nu f$ is surjective. Suppose for contradiction that there is a $y \in \nu Y$ whose fibre under $\nu f$ is disconnected, i.e.~$\nu f^{-1}(y) = A \cup B$ with $A$ and $B$ disjoint closed subsets. Define a function $g$ on $A\cup B$ to $\mathbb{C}$ which sends $A$ to $0$ and $B$ to $1$. By the Tietze Extension Theorem we can extend $g$ to a bounded continuous function $g' : hX \rightarrow \mathbb{C}$. In particular, $g'|_X: X \rightarrow \mathbb{C}$ must be slowly oscillating. Since $hf$ is a closed map, and $hY$ is compact, we may choose a open neighbourhood $D$ of $y$ in $hY$ such that 
$$
hf^{-1}(D) \subseteq g'^{-1}(B(0, 1/4) \cup B(1, 1/4))
$$
Let $A' = g'^{-1}(B(0, 1/4)) \cap hf^{-1}(D) \cap  X$ and $B' = g'^{-1}(B(1, 1/4)) \cap hf^{-1}(D) \cap X$. Since $y \in \overline{hf(A')} \cap \overline{hf(B')} \cap \nu Y$, by Lemma \ref{int}, there is an $S > 0$ such that, for every bounded set $K$ in $X$, there is a pair $a_K, b_K$ of points in $X$ with $a_K \in A' \setminus K$, $b_K \in B' \setminus K$ and $d(f(a_K), f(b_K)) \leq S$.  We now use the coarse monotone property of $f$ to get that each $a_K$ and $b_K$ are $\mathcal{W}$-connected inside an element of $f^{-1}(\mathcal{V})$ for some uniformly bounded families $\mathcal{W}$ in $X$ and $\mathcal{V}$ in $Y$. Since $g'$ is slowly oscillating, there is some bounded set $K'$ in $X$ such that, for $K' \subseteq K$, there must be an element of the chain from $a_K$ to $b_K$, say $c_K$, such that $g'(c_K) \in \mathbb{C} \setminus (B(0, 1/3) \cup B(1, 1/3))$. It follows that the closure of $C = \{ c_K \mid K' \subseteq K\}$ in $X$ intersected with $\nu X$ does not intersect $g'^{-1}(B(0, 1/4))$, and thus that $\overline{hf(C)}$ does not intersect $\overline{f(A')}$. But the set $\{ f(c_K), f(a_K) \}$ is uniformly bounded, so the closure of the $f(c_K)$ and the $f(a_K)$ intersect in $\nu Y$ by Lemma \ref{int}, which is a contradiction.

$(\Leftarrow)$ By Proposition \ref{coronaeq}, $f$ is coarsely surjective. Suppose $f$ is not coarsely monotone. This means that there is a uniformly bounded cover $\mathcal{U}$ of $Y$ such that for every integer $m > 0$, there is an element $U_{m} \in \mathcal{U}$ such that $f^{-1}(U_{m})$ is not contained in an $m$-component of $f^{-1}(B(U_{m}, m))$. Moreover, we may assume that the distance from each $f^{-1}(B(U_{m}, m))$ to all the previous $f^{-1}(B(U_i, i))$ tends to infinity as $m \rightarrow \infty$. Define a map $g$ on the union of the $f^{-1}(B(U_m, m))$ to $\{0,1\}$ such that each $m$-component of $f^{-1}(B(U_m, m))$ is mapped to a single value, and such that $g$ is surjective on each $f^{-1}(U_m)$. It is easy to see that this is a slowly oscillating function, so that it extends to a slowly oscillating function $g': X \rightarrow [0,1] \subseteq \mathbb{C}$. The map $g'$ extends to a continuous map $g'': hX \rightarrow \mathbb{C}$. Let $A = g''^{-1}(0) \cap X$ and $B = g''^{-1}(1) \cap X$ and let $A'$ and $B'$ be the intersection of the union of the $f^{-1}(U_m)$ with $A$ and $B$ respectively. By Lemma \ref{int}, the closures of $f(A')$ and $f(B')$ intersect on $\nu Y$. Suppose $y \in \overline{f(A')} \cap \overline{f(B')}$. We claim that $\overline{A}$ and $\overline{B}$ cover $hf^{-1}(y) \subseteq \nu X$. Indeed, suppose that $hf(x) = y$ with $g''(x) \notin \{0,1\}$. Pick a subset $S$ of $X$ such that $g''(S) \subseteq B(g'(x), \varepsilon)$ for $\varepsilon > 0$ small, and $x \in \overline{S}$. Since $\overline{f(A')}$ and $\overline{f(S)}$ intersect, by Lemma \ref{int}, there must be a $R > 0$ such that $B(f(A'), R) \cap B(f(S), R)$ is unbounded. But then $S$ must intersect some $f^{-1}(B(U_m, m))$, which is a contradiction, since $g''$ was defined to be $0$ or $1$ on these sets. Thus the fibre of $y$ under $hf$ is covered by the disjoint closed sets $\overline{A}$ and $\overline{B}$, so it cannot be connected. 
\end{proof}

As was previously mentioned, the Higson corona can also be defined for arbitrary ls-spaces. Unfortunately, Theorem \ref{corona} no longer holds in this more general context, as the following example shows.

\begin{Example}
Let $X$ be $\mathbb{N}$ with the usual metric ls-structure, and let $Y$ be the set $\mathbb{N}$ with the \textbf{universal bounded geometry structure}~\cite{Roe}, i.e.~wherein the uniformly bounded families are those families $\mathcal{U}$ such that $\{ |U| \mid U \in \mathcal{U} \}$ is bounded and $\mathcal{U}$ has finite point multiplicity. Recall from \cite{Roe} that the Higson corona of $Y$ is the one-point space, and recall from \cite{WeiConn} that $\mathbb{N}$ has a connected Higson corona. Let $f: X \rightarrow Y$ be the identity set map. It is clearly coarse, but is not coarsely monotone. Indeed, consider the uniformly bounded family $\mathcal{V} = \{ \{n^2, (n+1)^2\} \mid n \in \mathbb{N} \}$ in $Y$. The family $f^{-1}(\mathcal{V})$ does not refine an $M$-connected family of subsets with bounded cardinality for any $M$, so $f$ is not monotone. Nonetheless, the induced map on Higson coronas sends $\nu X$ to a single point, and is consequently (classically) monotone. 
\end{Example}

In fact, it is easy to see that in the above example, $f$ is coarsely light, but the induced map $\nu f$ is not classically light. Thus we cannot expect an equivalence of the form of Theorem \ref{corona} for coarsely/classically light maps for general ls-spaces. 

\begin{Proposition}
Let $f: X \rightarrow Y$ be a coarse map between proper metric spaces. If $\nu f$ is light, then $f$ is coarsely light.
\end{Proposition}
\begin{proof}
Factorize $f = f'e$ with $f'$ coarsely light and $e$ coarsely monotone. By Theorem \ref{corona}, $\nu e$ is monotone. Since $\nu f$ is light and factors through $\nu e$, $\nu e$ must be a homeomorphism. But by Proposition \ref{coronaeq}, this implies that $e$ is a coarse equivalence, so that $f$ is light as required.
\end{proof}

\begin{Question}
 Suppose $f: X \rightarrow Y$ is a coarse map between proper metric spaces. Is $\nu f$ light if $f$ is coarsely light?
\end{Question}

\section{Asymptotic dimension and exactness} \label{secPreserve}

In this section we investigate the permanence of some coarse properties under coarsely light maps. Since we have already seen that coarsely $n$-to-$1$ maps are coarsely light, these results generalize some results of Dydak-Virk \cite{DyVi} obtained for coarsely $n$-to-$1$ maps.

\begin{Proposition}\label{lightasdim}
Suppose $f:X\to Y$ is an ls-continuous map between ls-spaces. If $f$ is coarsely light, then the asymptotic dimension of $X$ is at most the asymptotic dimension of $Y$.
\end{Proposition}
\begin{proof}
Suppose $\asdim(Y)\leq k < \infty$ and $\mathcal{U}$ is a uniformly bounded cover of $X$. Pick a uniformly bounded cover $\mathcal{V}$ of $Y$ coarsening $f(\mathcal{U})$ that has multiplicity at most $k+1$. Consider the family $\mathcal{W}$ consisting of the $\mathcal{U}$-components of elements of $f^{-1}(\mathcal{V})$. Since $f$ is light, $\mathcal{W}$ is uniformly bounded. Moreover, it coarsens $\mathcal{U}$, and its multiplicity is at most $k+1$. 
\end{proof}

The following definition generalizes the concept of exactness from metric spaces (as introduced by Dadarlat-Guentner \cite{DaGu})
to arbitrary ls-spaces. For an index set $S$, let $\Delta(S)$ denote the set of formal linear combinations 
$$
\sum_{s \in S} a_s \cdot s 
$$
such that $a_s \in [0,1]$ for each $s$, $a_s = 0$ for all but finitely many $s$, and $\sum a_s = 1$. We will equip $\Delta(S)$ with the $l^1$ metric. The \textbf{star} of a vertex $s \in S$ is the set of all elements of $\Delta(S)$ with $a_s \neq 0$. By a \textbf{partition of unity} on a set $X$, we mean a map $\phi: X \rightarrow \Delta(S)$ for some set $S$. Recall that the \textbf{mesh} of a family $\mathcal{U}$ of subsets of a metric space $X$ is defined as follows
$$
\mathsf{mesh}(\mathcal{U}) = \mathsf{sup}\{ \mathsf{diam}(U) \mid U \in \mathcal{U} \}.
$$
In particular, the family $\mathcal{U}$ is uniformly bounded if and only if it has finite mesh.

\begin{Definition}
 A large scale space $X$ is \textbf{exact} if for each uniformly bounded cover
 $\mathcal{U}$ of $X$ and each $\epsilon > 0$ there is a partition of unity
 $\phi:X\to \Delta(S)$ such that point-inverses of stars of vertices form a uniformly bounded cover of $X$
 and the mesh of $\phi(\mathcal{U})$ is smaller than $\epsilon$.
\end{Definition}

\begin{Theorem}\label{ExactnessThm}
 Suppose $f:X\to Y$ is a large scale continuous map between ls-spaces.
 If $f$ is coarsely light and $Y$ is exact, then $X$ is exact.
\end{Theorem}
\begin{proof}
 Suppose
 $\mathcal{U}$ is a uniformly bounded cover of $X$
and $\epsilon > 0$. Choose a partition of unity $\phi:Y\to \Delta(S)$ such that point-inverses of stars of vertices form a uniformly bounded cover of $Y$ and the mesh of $\phi(f(\mathcal{U}))$ is smaller than $\epsilon$. Consider the family $J$ of $\mathcal{U}$-components of point-inverses of stars of vertices of the partition of unity $\phi\circ f:X\to \Delta(S)$. Create a new partition of unity $\psi:X\to \Delta(J)$ as follows:
  $$\psi(x)=\sum\limits_{j\in J} a_j\cdot j,$$
where $a_j\ne 0$ only if $x$ belongs to $j$, in which case it equals the coefficient of $\phi(f(x))$ at the corresponding $s\in S$. If $x, y$ belong to $U\in \mathcal{U}$, then they always belong to the same $\mathcal{U}$-component, so the distance from $\psi(x)$ to $\psi(y)$ is less than $\epsilon$. Since $f$ was coarsely light, the family of point-inverses of stars of vertices (that is, the family $J$) is uniformly bounded.
\end{proof}

\begin{Corollary} \label{PropertyA}
 Suppose $f:X\to Y$ is a large scale continuous map between metric spaces of bounded geometry.
 If $f$ is coarsely light and $Y$ has Property A \cite{Yu00}, then $X$ has Property A.
\end{Corollary}
\begin{proof}
As shown in \cite{DaGu}, a metric space of bounded geometry has Property A if and only if it is exact.
\end{proof}

\section{Groups}
In this section we make some remarks on the case when the ls-spaces involved are groups and the maps are group homomorphisms. Let $X$ be a (discrete) group. Following \cite{DH} we equip $X$ with the ls-structure consisting of all refinements of covers of the form
$$
\{ x\cdot F \mid x \in X\}
$$
where $F$ is a finite subset of $X$. If $X$ is countable, then this ls-structure coincides with the ls-structure arising from any proper left-invariant metric on $X$ (see \cite{JSmith}). In particular, if $X$ is finitely generated, then this ls-structure coincides with that induced by the word-length metric associated to any choice of finite generating set (see for example \cite{NowakYu}). Clearly any group homomorphism $f: X \rightarrow Y$ is ls-continuous with respect to the ls-structures on $X$ and $Y$.

\begin{Lemma}\label{asdimgroups}
Let $X$ be a group. Then $X$ has asymptotic dimension zero if and only if it is locally finite (i.e.~every finitely generated subgroup is finite). 
\end{Lemma}
\begin{proof}
$(\Rightarrow)$: Let $F$ be a finite subset, and $\mathcal{U} = \{ x \cdot F \mid x \in X \}$ the corresponding cover. Note that any element of $\langle F \rangle$ is $\mathcal{U}$-connected to the identity element $e$, so that $\langle F \rangle$ is contained in the $\mathcal{U}$-component of $e$, which by hypothesis is bounded and hence finite.

$(\Leftarrow)$: Let $\mathcal{U} = \{ x\cdot F \mid x \in X\}$ be a cover where $F$ is a finite subset. Notice that $x \mathcal{U} y$ for elements $x, y \in X$ if and only if $x^{-1}y \in F\cdot F$, from which it follows inductively that if $x$ and $y$ are $\mathcal{U}$-connected, then $x^{-1}y \in \langle F \rangle$. By assumption, $\langle F \rangle$ is finite, so the $\mathcal{U}$-component of $e$ is finite. Every other $\mathcal{U}$-component is a left translation of this component, so the family of $\mathcal{U}$-components is uniformly bounded, as required.
\end{proof}

Note that the above lemma was proved for countable groups in \cite{JSmith}; the proof given above is a straightforward adaptation of the proof found there. We will need the following lemma, based on the Finite Union Theorem in \cite{BellDran}. For a subspace $A \subseteq X$ of an ls-space $X$ and a family $\mathcal{U}$ of subsets of $X$, we write $\mathcal{U}|_A$ to mean the family $\{A \cap U \mid U \in \mathcal{U} \}$.

\begin{Lemma} \label{finiteunion}
Let $X$ be an ls-space. If $X = A \cup B$ for subsets $A$ and $B$ of $X$, and $A$ and $B$ each have asymptotic dimension zero as a subspace, then $X$ has asymptotic dimension zero.
\end{Lemma}
\begin{proof}
Let $\mathcal{U}$ be a uniformly bounded cover of $X$. Let $\mathcal{V}_A$ and $\mathcal{V}_B$ be the families of $\mathcal{U}|_A$-components and $\mathcal{U}|_B$-components of $A$ and $B$ respectively. Then $\mathcal{V}_A$ and $\mathcal{V}_B$ are uniformly bounded by hypothesis. If $a \in A$ is $\mathcal{U}$-connected to $a' \in A$, then it is easy to see that $a$ is in the same $\st(\st(\mathcal{V}_B, \mathcal{U})|_A, \mathcal{V}_A)$-component of $A$ as $a'$. Note that the family $\mathcal{W}_1$ of $\st(\st(\mathcal{V}_B, \mathcal{U})|_A, \mathcal{V}_A)$-components of $A$ is uniformly bounded as a family in $X$. Using this and similar arguments one can construct a uniformly bounded family $\mathcal{W}$ in $X$ which coarsens the family of $\mathcal{U}$-components of $X$, which gives the required result.
\end{proof}

We will also need the following generalization of Corollary 1.19 in \cite{Roe}.

\begin{Lemma} \label{groupembed}
Let $h: A \rightarrow X$ be an inclusion of a subgroup $A$ into a group $X$. Then $f$ is a coarse embedding.
\end{Lemma}
\begin{proof}
Let $F$ be a finite subset in $X$. Pick a set of representatives $S$ for the left cosets of $A$ in $X$, and let $T$ be the subset of $S$ consisting of all those $s \in S$ such that $F \cap s A \neq \varnothing$. Clearly $T$ is finite. Let $F$ be the finite set $F ' = (\bigcup_{t \in T} t^{-1} \cdot F) \cap A$. If $a$ is an element of $A$ with $a = xf \in x \cdot F$ for some $x \in X$, $f \in F$, then $f = x^{-1}a \in F$, so we may pick a $t \in T$ in the same left coset of $A$ as $x^{-1}$. Then $a = xtt^{-1}f \in xt \cdot F'$ with $xt \in A$, and hence also $t^{-1}f \in A$. Thus we have that $\{ x \cdot F \cap A \mid x \in X\}$ refines the uniformly bounded family $\{ a \cdot F' \mid a \in A\}$ as required. 
\end{proof}

We are now ready to present a characterisation of those group homomorphisms which are coarsely light as maps between ls-spaces.

\begin{Proposition} 
Let $f: X \rightarrow Y$ be a group homomorphism. Then $f$ is coarsely light if and only if $\mathsf{ker}(f)$ has asymptotic dimension zero, or equivalently, if $\mathsf{ker}(f)$ is locally finite. 
\end{Proposition}
\begin{proof}
Note that by Lemma \ref{groupembed}, we can either consider $\mathsf{ker}(f)$ as a group itself or as a subspace of $X$ since the ls-structure is the same in each case.

$(\Rightarrow)$ This follows from Proposition \ref{lightequiv} since the set consisting only of the identity in $Y$ has asymptotic dimension zero.

$(\Leftarrow)$ By Lemma \ref{groupembed}, the inclusion of the image of $f$ into $Y$ is a coarse embedding, and hence coarsely light. Since coarsely light maps are closed under composition, it is sufficient to consider the case when $f$ is surjective. Let $F$ be a finite set in $Y$. Then $f^{-1}(F)$ is a union of $|F|$ copies of the kernel of $f$, so by Lemma \ref{finiteunion}, $f^{-1}(F)$ has asymptotic dimension zero. Since
$$
\{x \cdot f^{-1}(F) \mid x \in X\} = f^{-1} (\{y \cdot F \mid y \in Y \})
$$ 
when $f$ is surjective, the family of inverse images of the family $\{y \cdot F \mid y \in Y \}$ is a family of left translates of $f^{-1}(F)$, and so satisfies $\asdim = 0$ uniformly. Thus by Lemma \ref{lightequiv}, $f$ is coarsely light.
\end{proof}

\begin{Corollary}
Let $f: X \rightarrow Y$ be a group homomorphism whose kernel is locally finite. Then $X$ has finite asymptotic dimension if $Y$ does. If both $X$ and $Y$ are countable, then $X$ has Property A if $Y$ does.
\end{Corollary}

\begin{Lemma}
Let $X$ be a group. Then $X$ is finitely generated if and only if $X$ is $\mathcal{U}$-connected for some uniformly bounded family $\mathcal{U}$. 
\end{Lemma}
\begin{proof}
If $X$ is finitely generated, then its ls-structure is generated by a word-length metric, under which $X$ is clearly $1$-connected. Conversely, suppose $X$ is $\mathcal{U}$-connected, where $\mathcal{U} = \{x \cdot F \mid x \in X\}$. We claim that $F$ generates $X$. Indeed, as in the proof of Lemma \ref{asdimgroups}, $x$ and $y$ are $\mathcal{U}$-connected if and only if $x^{-1}y \in \langle F \rangle$, so that in particular, every $x \in X$, being in the $\mathcal{U}$-component of the identity, is in $\langle F \rangle$.
\end{proof}

\begin{Corollary}
A group $X$ is finitely generated if and only if the unique map from $X$ to the trivial group is coarsely monotone, and locally finite if and only if the unique map from $X$ to the trivial group is coarsely light.
\end{Corollary}

\end{document}